\documentclass[a4paper,11pt,reqno,twoside]{amsart}

\usepackage{amsmath}
\usepackage{amsfonts}
\usepackage{amssymb}
\usepackage{amsthm}
\usepackage{color}
\usepackage{ifpdf}
\usepackage{array}
\usepackage{url}
\usepackage{multirow,bigdelim}
\usepackage{ascmac}

\numberwithin{equation}{section}

\newtheorem{theorem}{Theorem}[section]

\newcommand*{\C}{\mathbb{C}}
\newcommand*{\R}{\mathbb{R}}
\newcommand*{\Z}{\mathbb{Z}}

\newcommand{\comment}[1]{}
\title[A probabilistic interpretation for central zeros of $L$-functions]%
      {A probabilistic interpretation for central zeros of $L$-functions in the Selberg class} 
\author[T. Nakamura]{Takashi Nakamura}
\author[M. Suzuki]{Masatoshi Suzuki}
\date{Version of \today}
\AtBeginDocument{%
\begin{abstract}
We show that central zeros of $L$-functions in the Selberg class 
have a probabilistic interpretation by stating an equivalence condition 
of the Riemann hypothesis for the $L$-functions 
in terms of infinitely divisible distributions.
\end{abstract}
\maketitle
}
\begin{document}

%
\section{Introduction} 
%

The study of central values or central zeros of $L$-functions 
is a significant subject in number theory. 
In fact, 
there are deep arithmetic reasons for central zeros of $L$-functions, 
represented by the Birch and Swinnerton-Dyer conjecture, 
except for trivial reasons such as the negative sign of the functional equation.  
But in addition, in this paper, 
we propose a new perspective on the study of central values of $L$-functions 
by showing that the central zeros have a probabilistic interpretation. 
\medskip

In probability theory, the L{\'e}vy processes are the most fundamental stochastic processes, 
and they are described by infinitely divisible distributions. 
As we will briefly review in Section \ref{section_2_0}, 
typical infinitely divisible distributions are 
Gaussian distributions with nonzero Gaussian covariance and 
compound Poisson distributions with zero Gaussian covariance.
In \cite{NaSu23}, 
we proved that an infinitely divisible distribution arises from 
the Riemann zeta-function $\zeta(s)$ 
if assuming the Riemann hypothesis and vice versa. 
In this case, the distribution is compound Poisson. 
For other probabilistic studies related to the Riemann zeta-function, 
see the introduction of  \cite{NaSu23}, the literature referenced therein, 
and also Section \ref{section_3}. 

The Selberg class $\mathcal{S}$, introduced by A. Selberg in 1992, 
is a large class of Dirichlet series 
including 
the Riemann zeta-function, 
Dirichlet $L$-functions, and $L$-functions of holomorphic cusp forms, 
and is expected to include all important zeta and $L$-functions in number theory 
(see Section \ref{section_2_1}).  
The purpose of this paper is to show that 
if we generalize the above relation between the infinitely divisible distribution 
and the Riemann zeta-function to $L$-functions in the Selberg class $\mathcal{S}$, 
then we obtain infinitely divisible distributions with nonzero Gaussian covariance. 
A significant nature of the Selberg class is that for any $F \in \mathcal{S}$, 
all nontrivial zeros of $F$ are conjectured to 
lie on the central line of the functional equation between $F(s)$ and $\overline{F(1-\bar{s})}$ (\eqref{s202} below). 
This conjecture is often called the {\it Grand Riemann Hypothesis} (GRH, for short). 
The zero at the central point $s=1/2$ is called the {\it central zero} and 
is the only real nontrivial zero of $F$ under the GRH. 
The main theorem (Theorem \ref{thm_2_1} below) asserts that 
if $F \in \mathcal{S}$ has the central zero and satisfies the GRH, 
then it generates an infinitely divisible distribution with nonzero Gaussian covariance. 
It is interesting that the properties of the stochastic object 
corresponding to $F$ change greatly depending on the presence or absence of the central zero. 
\bigskip

In the following, we first review 
the definition of infinitely divisible distributions in Section \ref{section_2_0} 
and 
the definition of the Selberg class in Section \ref{section_2_1}. 
Then, we  state the precise statement and proof of the main results in Section \ref{section_2_2} 
and prove them in Section \ref{section_2_3}. 
Furthermore, we present concrete examples of $F \in \mathcal{S}$ having central zeros 
and make some remarks on infinitely divisible distributions and L\'evy processes 
constructed by zeta or $L$-functions belonging to the Selberg class in Section \ref{section_3}.

\section{Statements and proofs of results}

\subsection{Infinitely divisible distributions} \label{section_2_0} 

A distribution $\mu$ on the real line is said to be {\it infinitely divisible}  
if there exists a distribution $\mu_n$ on the real line 
such that $\mu$ is the $n$-hold convolution of $\mu_n$,  
$\mu=\mu_n \ast \dots \ast \mu_n$, 
for every positive integer $n$. 
For any infinitely divisible distribution $\mu$, 
there exists a triplet $(a,b,\nu)$ 
consisting of $a \in \R_{\geq 0}$, $b \in \R$ 
and a measure $\nu$ on $\R$ 
such that the characteristic function 
$
\widehat{\mu}(t)=\int_{-\infty}^{\infty} e^{itx} \mu(dx)
$
has the L{\'e}vy--Khintchine formula 
\begin{equation} \label{LK_1}
\widehat{\mu}(t) = \exp\left[
- \frac{1}{2} a t^2 + i b t 
+ \int_{-\infty}^{\infty}\left(
e^{it\lambda}  - 1 - \frac{it\lambda}{1+\lambda^2} 
\right) \nu(d \lambda)
\right], 
\end{equation}
\begin{equation} \label{LK_2}
\nu(\{0\})=0, \quad \int_{-\infty}^{\infty} \min(1,\lambda^2) \,\nu(d\lambda) < \infty
\end{equation}
(\cite[Theorem 8.1 and Remark 8.4]{Sa99}). 
In the triplet $(a,b,\nu)$ in \eqref{LK_1}, 
$a$, $b$, and $\nu$ are called the {\it Gaussian covariance}, {\it center}, 
and {\it L{\'e}vy measure} of $\mu$, respectively. 
If the L{\'e}vy measure $\nu$ satisfies $\int_{|\lambda|\leq 1} |\lambda|\,\nu(d\lambda)<\infty$, 
then \eqref{LK_1} can be rewritten as
\begin{equation} \label{LK_3}
\widehat{\mu}(t) = \exp\left[
- \frac{1}{2} a t^2 + i b_0 t 
+ \int_{-\infty}^{\infty}(e^{it\lambda}  - 1) \, \nu(d \lambda)
\right]
\end{equation}
(\cite[(8.7)]{Sa99}), where $b_0 \in \R$ is called the {\it drift} of $\mu$.
A function of the form \eqref{LK_1} or \eqref{LK_3} 
is often called an infinitely divisible characteristic function. 
If $\widehat{\mu_1}(t)$ and $\widehat{\mu_2}(t)$ are infinitely divisible characteristic functions, 
then the product $\widehat{\mu_1}(t)\widehat{\mu_2}(t)$ is also an infinitely divisible characteristic function 
(\cite[Lemma 7.4]{Sa99}). 
The L{\'e}vy measure $\nu$ is zero if $\mu$ is a Gaussian distribution. 
In contrast, the Gaussian covariance $a$ and drift $b_0$ in (\ref{LK_3}) are zero 
if $\mu$ is a compound Poisson distribution (\cite[Examples 8.5]{Sa99}). 
The difference between these two extreme distributions is rather evident 
if we visualize the L{\'e}vy process associated with $\mu$.
\bigskip 

The simplest infinitely divisible characteristic function is written as (\ref{LK_1}) 
with $a=0$, $b \not=0$, and $\nu=0$. 
Then the corresponding L\'evy process is 
nothing but a linear function (\cite[p.~109 and Fig.~2.1]{EbKa19}).
The L\'evy process constructed by an infinitely divisible characteristic function 
with $a\not=0$, $b=b_0=0$, and $\nu=0$ in (\ref{LK_1}) or (\ref{LK_3}) 
is called Brownian motion, which is used to model particle movements 
(\cite[p.~113 and Fig.~2.6]{EbKa19} or \cite[Corollary 11.8]{Sa99}). 
In addition, the L\'evy process constructed by an infinitely divisible characteristic function 
with $a=b_0=0$ and $\nu\not=0$ in (\ref{LK_3}) is called a compound Poisson process, which plays important role in risk theory (\cite[p.~111 and Fig.~2.4]{EbKa19} or \cite[p.~19]{Sa99}). 
In short, these three L\'evy processes are characterized by straight line, isolated jumps, 
and continuous wiggling, respectively.

\subsection{The Selberg class} \label{section_2_1} 

The Selberg class $\mathcal{S}$ consists of the Dirichlet series 
\begin{equation} \label{s201}
F(s) = \sum_{n=1}^{\infty} \frac{a_F(n)}{n^s}
\end{equation}
satisfying the following five axioms: 
\begin{enumerate}
\item[(S1)] The Dirichlet series \eqref{s201} converges absolutely if $\Re(s)>1$. 
\item[(S2)] There exists an integer $m \geq 0$ 
such that $(s - 1)^mF(s)$ extends to an entire function of finite order. 
The smallest $m$ is denoted by $m_F$. 
\item[(S3)] $F$ satisfies the functional equation
\begin{equation} \label{s202}
\xi_F(s)=\omega \, \overline{\xi_F(1-\bar{s})},
\end{equation}
where
\begin{equation} \label{s203}
\aligned 
\xi_F(s) 
& = s^{m_F}(s-1)^{m_F}Q^s \prod_{j=1}^{r} \Gamma(\lambda_{j} s + \mu_{j}) \, F(s), 
\endaligned 
\end{equation}
$\Gamma(s)$ is the gamma function,  
and $r \geq 0$, $Q > 0$, $\lambda_{j} > 0$, $\mu_{j} \in \C$ with $\Re(\mu_{j})\geq 0$, 
$\omega \in \C$ with  $|\omega| = 1$ are parameters depending on $F$. 
\item[(S4)] For every $\varepsilon > 0$, $a_F(n) \ll_\varepsilon n^\varepsilon$. 
\item[(S5)] 
$\log F(s) = \sum_{n=1}^{\infty} b_F(n)\,n^{-s}$,  
where $b_F(n) = 0$ unless $n = p^m$ with $m \geq 1$, and $b_F(n) \ll n^\theta$ for some $\theta < 1/2$.
\end{enumerate}
The class $\mathcal{S}$ is a multiplicative semigroup by definition, 
that is, if $F_1$ and $F_2$ belong to $\mathcal{S}$, 
then the product $F_1F_2$ also belongs to $\mathcal{S}$. 
The Riemann zeta-function $\zeta(s)$ is an element of the Selberg class. 
In fact, 
(S1) and (S4) hold by $a_\zeta(n)=1$, 
(S2) holds for $m_\zeta=1$ because $\zeta(s)$ has a simple pole at $s=1$,  
(S3) holds for $r=1$, $Q=\pi^{-1/2}$, $\lambda_1=1/2$, $\mu_1=0$, 
and $\omega=1$ 
by $\xi_\zeta(s)=\xi_\zeta(1-s)$ (\cite[(2.1.13)]{Tit86}), and 
(S5) holds by $b_\zeta(1)=0$ and $b_\zeta(n)=\Lambda(n)/\log n$ for $n \geq 2$ (\cite[(1.1.9)]{Tit86}). 

From (S3) and (S5), 
every $F \in \mathcal{S}$ has no zeros outside the critical strip $0 \leq \Re(s) \leq 1$ 
except for zeros in the left half-plane $\Re(s)  \leq 0$ located at poles of the involved gamma factors. 
A zero lies in the critical strip is called a {\it nontrivial zero} 
if and only if it is a zero of the entire function $\xi_F(s)$ in \eqref{s203}. 
There are infinitely many nontrivial zeros unless $F \equiv 1$. 
The Grand Riemann Hypothesis for $F \in \mathcal{S}$ 
mentioned in the introduction is the assertion that 
all nontrivial zeros of $F(s)$ lie on the line $\Re(s)=1/2$, 
which is equivalent to all the zeros of $\xi_F(1/2-iz)$ being real. 
The survey \cite{Pe04, Pe05} and its sequel are good introductions to 
the theory of the Selberg class. 

\subsection{Main Results} \label{section_2_2} 

As a generalization of $g_\zeta(t)$ used in \cite{NaSu23} 
for associating $\zeta(s)$ to an infinitely divisible distribution under the Riemann hypothesis, 
we define the function $g_F(t)$ on the real line for $F \in \mathcal{S}$ by 
\begin{equation} \label{s204} 
g_F(t):= 
- \frac{m_0}{2}\,t^2
+ iB_F t
 + \sum_{\gamma\not=0} m_\gamma \frac{e^{-i\gamma t}-1}{\gamma^2}, 
\quad B_F := i\frac{\xi_F^\prime}{\xi_F}\left(\frac{1}{2}\right)
\end{equation}
referring to \cite[(1.7)]{NaSu23}, 
where $\gamma$ runs through all nonzero zeros of $\xi_F(1/2-iz)$ without counting multiplicity 
and $m_\gamma$ is the multiplicity of $\gamma$. 
Note that $m_0$ is the multiplicity of the central zero and this term is separated. 
The right-hand side of \eqref{s204} converges absolutely and uniformly 
on any compact subsets of $\R$, 
as shown similarly to the proof of \cite[Lemma 2.1]{NaSu23}, 
since $\xi_F(s)$ is an entire function of order one (\cite[p.29]{Pe05}). 
Therefore, $g_F(t)$ is a continuous function on the real line. 
Using $g_F(t)$, the main results are stated as follows. 

\begin{theorem} \label{thm_2_1}
Let $g_F(t)$ be the function on the real line defined by \eqref{s204} for $F \in \mathcal{S}$. 
Then, the following two statements are equivalent: 
\begin{enumerate}
\item the Grand Riemann Hypothesis for $F$ is true; 
\item the function $\exp(g_F(t))$ on the real line is 
a characteristic function of 
an infinitely divisible distribution on the real line. 
\end{enumerate}
\end{theorem}

Note that Theorem \ref{thm_2_1} (2) is consistent with 
the multiplicative property of infinitely divisible characteristic functions, 
since $g_{F_1F_2}(t)=g_{F_1}(t)+g_{F_2}(t)$ for $F_1$, $F_2 \in \mathcal{S}$
by definitions \eqref{s203} and \eqref{s204}. 

\begin{theorem} \label{thm_2_2} 
Let $g_F(t)$ be the function on the real line defined by \eqref{s204} for $F \in \mathcal{S}$. 
Assume that the GRH for $F$ is true and define 
\begin{equation} \label{Levy_2}
\nu_F(d\lambda) = \sum_{\gamma\not=0} \frac{m_\gamma}{\gamma^2} \, \delta_{-\gamma}(d\lambda),
\end{equation}
where the sum and $m_\gamma$ have the same meaning as in \eqref{s204}.
Then the L{\'e}vy--Khintchine formula \eqref{LK_3} 
of $\exp(g_F(t))$ holds for the triplet 
$(a,b_0,\nu)=(m_0,B_F,\nu_F)$. 
\end{theorem}

According to Theorem \ref{thm_2_2}, 
when $F \in \mathcal{S}$ satisfies the GRH, 
if the central value of $F$ is nonzero, 
then $\exp(g_F(t))$ is the characteristic function of a compound Poisson distribution. 
Conversely,  if $F$ has a central zero, 
then $\exp(g_F(t))$ is the characteristic function of the convolution product 
of a Gaussian distribution and a compound Poisson distribution, 
and the multiplicity $m_0$ of the central zero 
is interpreted as the Gaussian covariance of the Gaussian distribution.  
Thus, the probabilistic property of $F$ described via $g_F$ changes 
depending on the presence of central zeros.

\subsection{Proofs of Theorems \ref{thm_2_1} and \ref{thm_2_2} } \label{section_2_3}
\begin{proof}[Proof of (1)$\Rightarrow$(2) in Theorem \ref{thm_2_1} and Proof of Theorem \ref{thm_2_2}] 
By definition \eqref{s204}, 
equality \eqref{Levy_2} defines a measure on $\R$ such that  
the formula 
\[
\aligned 
\exp(g_F(t)) 
& = \exp\left[ - \frac{m_0}{2}\,t^2
+ iB_F t + \int_{-\infty}^{\infty}
(e^{it\lambda}  - 1) \nu_F(d \lambda) \right] 
\endaligned 
\]
holds on $\R$, $\nu_F(\{0\})=0$ and 
$\int_{|\lambda| \leq 1}|\lambda|\nu_F(d\lambda)<\infty$, 
since all zeros $\gamma$ of $\xi_F(1/2-iz)$ are real by assumption (1).  
Moreover, $\int_{-\infty}^{\infty}\min(1,\lambda^2)\,\nu_F(d\lambda)<\infty$ 
by the convergence of $\sum_{\gamma\not=0} m_\gamma |\gamma|^{-2}$, 
which follows from the fact that $\xi_F(s)$ is an entire function of order one. 
Hence, there exists an infinitely divisible distribution $\mu$ 
whose characteristic function is given by $\exp(g_F(t))$ 
with the characteristic triplet $(a,b_0,\nu)=(m_0,B_F,\nu_F)$ 
by \cite[Theorem 8.1 (iii)]{Sa99}. 
\end{proof}

\begin{proof}[Proof of (2)$\Rightarrow$(1) in Theorem \ref{thm_2_1}] 
We have $\exp(g_F(t))=\widehat{\mu}(t)$ for some infinitely divisible 
distribution $\mu$ on the real line by assumption (2). 
Therefore, $\Re(g_F(t))$ is nonpositive on $\R$, since $|\widehat{\mu}(t)| \leq 1$. 
Next, we show that $F$ has no real zeros other than $s=1/2$ in a few steps.
First, for $F \in \mathcal{S}$, we define 
$F^\ast(s):=\overline{F(\bar{s})}$. 
Then, $F^\ast \in \mathcal{S}$ with 
$a_{F^\ast}(n)=\overline{a_F(n)}$, 
$m_{F^\ast}=m_F$, 
$\omega_{F^\ast}=\overline{\omega_F}$, 
$r_{F^\ast}=r_F$, 
$Q_{F^\ast}=Q_F$, 
$\lambda_{F^\ast,j}=\lambda_{F,j}$, and 
$\mu_{F^\ast,j}=\overline{\mu_{F,j}}$ 
if we write $\omega$, $Q$, $r$, $\lambda_j$, and $\mu_j$ in (S3) 
as $\omega_F$, $Q_F$, $r_F$, $\lambda_{F,j}$, and $\mu_{F,j}$, respectively. 
In particular, 
\begin{equation} \label{s206}
\xi_{F^\ast}(s) = \overline{\xi_F(\bar{s})}. 
\end{equation}
Further, the GRH for $F$ is equivalent to the GRH for $FF^\ast$ 
by \eqref{s202} and \eqref{s206}. 

By \eqref{s202} and \eqref{s206}, 
the mapping $\gamma \mapsto -\overline{\gamma}$ 
defines a bijection from the set of all zeros of $\xi_F(1/2-iz)$ 
to the set of all zeros of  $\xi_{F^\ast}(1/2-iz)$ with multiplicity, 
and $B_F=-B_{F^\ast}\,(\in \R)$. 
Applying these to \eqref{s204}, 
we get $\overline{g_F(t)} = g_{F^\ast}(t)$. 
Moreover, $F$ has a zero in $(1/2,\infty)$ 
if and only if $F^\ast$ has a zero in $(1/2,\infty)$ 
by \eqref{s202} and \eqref{s206}. 

Therefore, the absence of the zeros of $F$ in $(1/2,\infty)$ 
implies the absence of the zeros of $FF^\ast$ in $(1/2,\infty)$. 
As mentioned in Section \ref{section_2_1}, 
real zeros of $F$ in $(1/2,\infty)$ lie in $(1/2,1]$ if they exist. 
Let $\beta\,(>1/2)$ be the maximum of such zeros. 
Then $\xi_F(1/2-iz)$ has the zero $\gamma_0:=i(\beta-1/2)$. 
Therefore $|g_F(t)|$ grows exponentially by \eqref{s204}. 
However, the L{\'e}vy--Khintchine formula \eqref{LK_1} 
shows that $|g_F(t)|=|\log \widehat{\mu}(t)|$ is at most polynomial growth. 
This is a contradiction, and therefore $F$ has no real zeros except for $s=1/2$. 

On the other hand, we show that 
\begin{equation} \label{s207}
\aligned 
2\int_{0}^{\infty} & \Re(g_F(t)) e^{izt} \, dt 
= \int_{0}^{\infty} (g_F(t)+g_{F^\ast}(t)) e^{izt} \, dt \\
& = \frac{1}{z^2} 
\left[ 
\frac{\xi_F^\prime}{\xi_F}\left(\frac{1}{2}-iz \right)
+
\frac{\xi_{F^\ast}^\prime}{\xi_{F^\ast}}\left(\frac{1}{2}-iz \right)
\right] 
 = \frac{1}{z^2} 
\frac{\xi_{FF^\ast}^\prime}{\xi_{FF^\ast}}\left(\frac{1}{2}-iz \right)
\endaligned 
\end{equation}
holds if $\Im(z)>1/2$. 
Taking the logarithmic derivative of Hadamard's factorization formula 
\[
\xi_F(1/2-iz) = e^{A_F-B_Fz} z^{m_0} \prod_{\gamma \not=0} \left[\left(1-\frac{z}{\gamma} \right)
e^{\frac{z}{\gamma}} \right]^{m_\gamma}, 
\]
and then substituting $z=0$, we have $B_F=i(\xi_F^\prime/\xi_F)(1/2)$. 
Using 
\begin{equation} \label{s208}
\frac{i}{z^2}  \left( \frac{1}{z-\gamma} + \frac{1}{\gamma} \right)
=
\int_{0}^{\infty} \frac{e^{-it\gamma}-1}{\gamma^2}\, e^{izt} \, dt, 
\quad \Im(z)>\Im(\gamma)
\end{equation}
\begin{equation} \label{s209}
\int_{0}^{\infty} t \, e^{izt} \, dt = -\frac{1}{z^2}, 
\qquad 
\int_{0}^{\infty} \frac{t^2}{2} \, e^{izt} \, dt = -\frac{i}{z^3}, \quad \Im(z)>0, 
\end{equation}
we obtain
\begin{equation} \label{s210}
\int_{0}^{\infty} g_F(t) \,e^{izt} \, dt = \frac{1}{z^2}\frac{\xi_F'}{\xi_F}\left(\frac{1}{2}-iz\right), 
\quad \Im(z)>1/2 
\end{equation}
in the same way as the proof of \cite[Lemma 2.2]{NaSu23}. 
Thus \eqref{s207} follows from \eqref{s210}.  

As a result of the discussion above, 
by the same argument as in the proof of the implication (2)$\Rightarrow$(1) 
in \cite[Theorem 1.1]{NaSu23},  
we obtain the GRH  for $FF^\ast$ from the integral formula \eqref{s207} 
using the nonpositivity of $\Re(g_F(t))$ on $\R$ 
and the absence of the zeros of $F$ in $(1/2,\infty)$. 
\end{proof}

\section{Concluding remarks} \label{section_3}

The Selberg class is graded by degree, 
and it is known that the only elements of degree less than two  
are the Riemann zeta-function $\zeta(s)$ and Dirichlet $L$-functions $L(s,\chi)$ 
(\cite{Pe04, Pe05}). 
For the Riemann zeta-function, it is well-known that $\zeta(1/2)\not=0$. 
For Dirichlet $L$-functions, it is conjectured that 
the central value $L(1/2,\chi)$ is never zero (\cite[Section 1]{CS}). 
Therefore, it is expected the degree of elements of $\mathcal{S}$ 
with central zeros be at least two. 
And indeed, there are the following examples of $F \in \mathcal{S}$ 
of degree two having central zeros ($m_0>0$) and no other real zeros. 
\medskip

Let $S_k$ be the space of holomorphic cusp forms 
for the full modular group $SL_2(\Z)$ (\cite[Section 2.1]{Mi89}). 
For each $f(z)=\sum_{n=1}^{\infty} a_n \exp(2\pi i nz)$ in $S_k$, 
the $L$-function $L(s,f)$ is defined 
by the Dirichlet series $\sum_{n=1}^{\infty} a_n n^{-s}$ (\cite[(4.3.11)]{Mi89}), 
analytically continued to $\C$, 
and satisfies the functional equation 
$\Lambda(s,f)=i^k \Lambda(k-s,f)$ 
with $\Lambda(s,f)=(2\pi)^{-s}\Gamma(s)L(s,f)$ (\cite[Corollary 4.3.7]{Mi89}). 
Moreover, the estimate $|a_n|=O(n^{(k-1)/2})$ holds (\cite[Theorem 4.5.17]{Mi89}). 
Therefore, $F_f(s):=L(s+(k-1)/2)$ satisfies (S1), (S2), (S3), and (S4). 
We have $F_f(1/2)=0$ by the functional equation if $k \equiv 2$ mod $4$. 
Moreover, if ${\rm dim}\,S_k=1$, $F_f$ satisfies (S5) by \cite[Theorem 4.5.16]{Mi89},  
because $S_k$ is closed under the action of Hecke operators (\cite[Theorem 4.5.4 (3)]{Mi89}). 
The spaces $S_{18}$, $S_{22}$, and $S_{26}$ are one dimensional by \cite[Corollary 4.1.4]{Mi89}. 
Hence, $F_f$ for $f \in S_{k}$ ($k=18, 22, 26$) are examples of $F \in \mathcal{S}$ 
satisfying $F(1/2)=0$. 
By the product formula of $L(s,f)$ in \cite[Theorem 4.5.16]{Mi89}, 
$F_f(s) \not=0$ for $s \in (1, \infty)$. 
Choosing explicit generators 
$\Delta E_6 \in S_{18}$,  $\Delta E_4 E_6 \in S_{22}$, $\Delta E_4^2 E_6 \in S_{26}$, 
we can numerically check that $F_f(s) \not=0$ for $s \in (1/2, 1]$ (cf. \cite[Section 4.1]{Mi89}), 
so $F_f(s)\not=0$ for $s \in \R \setminus \{1/2\}$ by the functional equation, 
where $\Delta$ is the modular discriminant and $E_k$ are Eisenstein series as usual. 
\medskip

The study on infinite divisibility of characteristic functions induced by zeta or $L$-functions has a long history. In \cite[p.~67]{GK68}, it is shown that the distribution generated by $\zeta(\sigma+it)/\zeta (\sigma)$ is compound Poisson on $\R$ 
and
\begin{equation*} 
\frac{\zeta(\sigma+it)}{\zeta (\sigma)} = \exp \left[ \int_{-\infty}^\infty (e^{it\lambda} - 1) N_{\sigma} (d\lambda) \right], \quad
N_{\sigma} (d\lambda) := \sum_{p \in \mathbb{P}} \sum_{r=1}^\infty \frac{p^{-r\sigma}}{r} \delta_{-r\log p} (d\lambda),
\end{equation*}
where $\mathbb{P}$ denotes the set of all prime numbers. 
Lin and Hu generalized this result for $\sum_{n=1}^\infty c(n) n^{-s}$ 
with completely multiplicative nonnegative coefficients $c(n)$ in \cite[Theorem 2]{LH01}. Furthermore, 
Aoyama and Nakamura 
consider multidimensional $\eta$-tuple $\phi$-rank compound Poisson zeta distributions on $\R^d$ 
(Aoyama and Nakamura, 
Multidimensional polynomial Euler products and infinitely divisible distributions on $\R^d$, 
Theorem 3.8, 
\newblock{\url{https://arxiv.org/abs/1204.4041}). 
All zeta distributions appeared in \cite{GK68}, \cite{LH01}, and 
Aoyama--Nakamura are compound Poisson, namely, 
the corresponding Gaussian covariance and drift are zero in (\ref{LK_3}).

However, if $F \in \mathcal{S}$ has a real zero at $s=1/2$, 
the corresponding Gaussian covariance is positive. 
Note that when $F$ is the Riemann zeta-function, 
the corresponding Gaussian covariance and drift are zero (\cite[Theorem 1.2]{NaSu23}). 
In \cite[Section 2.4]{EbKa19}, there are many L\'evy  processes 
that may be related to modeling real-world phenomena such as logarithmic asset prices.
By using $L$-functions in the Selberg class, 
we can treat such wider L\'evy processes.

\medskip

\noindent
{\bf Acknowledgments}~
The first and second authors were supported by JSPS KAKENHI 
Grant Number JP22K03276 and JP17K05163--JP23K03050, respectively. 
This work was also supported by the Research Institute for Mathematical Sciences,
an International Joint Usage/Research Center located in Kyoto University.

%

%
\bigskip 

\noindent
Takashi Nakamura,\\[5pt]
Institute of Arts and Sciences, Noda Division, \\
Tokyo University of Science, \\ 
2641 Yamazaki, Noda-shi, \\ 
Chiba-ken
278-8510, Japan  \\[2pt]
Email: {\tt nakamuratakashi@rs.tus.ac.jp}
\bigskip

\noindent
Masatoshi Suzuki,\\[5pt]
Department of Mathematics, \\
Tokyo Institute of Technology \\
2-12-1 Ookayama, Meguro-ku, \\
Tokyo 152-8551, Japan  \\[2pt]
Email: {\tt msuzuki@math.titech.ac.jp}
\end{document}